\documentclass[12pt,reqno]{amsart}

 \pdfoutput=1

\usepackage{amsmath}
\usepackage{amsthm}
\usepackage{amssymb}
\usepackage{mathrsfs}
\usepackage{enumitem}

\numberwithin{equation}{section}
\newtheorem{theor}{Theorem}

\newtheorem{coro}[theor]{Corollary}

\newtheorem{lemma}[theor]{Lemma}

\theoremstyle{definition}
\newtheorem{defi}[theor]{Definition}

\theoremstyle{remark}

\DeclareMathOperator{\supp}{supp}

\DeclareMathOperator{\conv}{conv}
\newcommand{\calF}{\mathcal F}
\newcommand{\calH}{\mathcal H}
\newcommand{\calX}{\mathcal X}
\newcommand{\Xe}{\mathcal X_\varepsilon}
\newcommand{\Ye}{\mathcal Y_\varepsilon}
\newcommand{\Xk}{\mathcal X_K}
\newcommand{\Yk}{\mathcal Y_K}
\newcommand{\PP}{\mathbb P}

\title{Banach spaces from a construction scheme}
\author{Fulgencio Lopez}
\address{Department of Mathematics\\ University of Toronto\\ Bahen Center 40 St. George St.\\ Toronto, Ontario M5S 2E4, Canada.}
\email{fulgencio.lopez@mail.utoronto.ca}


\begin{document}

\begin{abstract}
We construct a Banach space $\mathcal X_\varepsilon$ with an uncountable $\varepsilon$-biorthogonal system but no uncountable $\tau$-biorthogonal system for $\tau<\varepsilon(1+\varepsilon)^{-1}$. In particular the space have no uncountable biorthogonal system. We also construct a Banach space $\mathcal X_K$ with an uncountable $K$-basic sequence but no uncountable $K'$-basic sequence, for $1\leq K'<K$. A common feature of these examples is that they are both constructed by recursive amalgamations using a single construction scheme.
\end{abstract}

\maketitle

\section{Introduction}

The class of nonseparable Banach spaces exhibit phenomena which are not present in the more studied class of separable Banach spaces. Some of the most striking differences were discovered recently by
J. Lopez-Abad and S. Todorcevic \cite{LAT} when they were developing forcing constructions of Banach spaces via finite-dimensional approximations. For example, it is shown in \cite{LAT} that for every $\varepsilon>0$ rational, there is a forcing notion $\PP_\varepsilon$ which forces a Banach space $\Ye$ with an uncountable $\varepsilon$-biorthogonal system and such that for every $0\leq\tau<\frac{\varepsilon}{1+\varepsilon}$, $\Ye$ has no uncountable $\tau$-biorthogonal system. They also showed (\cite[Theorem 6.4]{LAT}) that for every constant $K>1$ there is a forcing notion $\PP_K$ which forces a Banach space  $\Yk$ with an uncountable $K$-basis yet for every $1\leq K'<K$, $\Yk$ has no uncountable
$K'$-basic sequences. Recall that none of these two phenomena can happen in the class of separable Banach spaces when, of course, we replace `uncountable' by `infinite'.

In \cite{todor}, S. Todorcevic introduced a notion of construction scheme of uncountable mathematical objects via finite approximation and simultaneous multiple amagamations. While the construction scheme $\mathcal{F}$ can be described relying only on ordinary axioms of set theory, their crucial properties of `capturing' (see the definition below) can only be provided using Jensen's combinatorial principle $\Diamond$: there are sets $A_\alpha\subset\alpha$ for every $\alpha<\omega_1$ such that for every $A$ subset of $\omega_1$ there are stationarily many $\alpha$'s with $A\cap\alpha=A_\alpha$.

The purpose of this note is to apply this construction scheme to the theory of nonseparable Banach spaces inspired by the forcing constructions of \cite{LAT}. In particular, we prove the following two results

\begin{theor}\label{theorem.1} Assume $\Diamond$.
Then for every $\varepsilon\in (0,1)\cap\mathbb Q$, there is a Banach space $\Xe$
with an uncountable $\varepsilon$-biorthogonal system but no uncountable $\tau$-biorthogonal system for every $0\leq\tau<\frac{\varepsilon}{1+\varepsilon}$.
\end{theor}

\begin{theor}\label{theorem.2} Assume
$\Diamond$.
Then for every constant $K> 1$, there is a Banach space $\Xk$ with a
$K$-basis of length $\omega_1$ 
but no uncountable $K'$-basic sequence for every $1\leq K'<K$.
\end{theor}

In each case the construction is based on a single rule of multiple amalgamation of a family of finite-dimensional Banach spaces indexed by $\mathcal{F}.$ This adds not only to the clarity over the corresponding forcing constructions but it also gives us Banach spaces that could be further easily analyzed.
In fact neither the construction nor the analysis of the corresponding examples require any expertise outside the Banach space geometry.

It is interesting to compare our examples with the corresponding examples in \cite{LAT}. Given an uncountable sequence of forcing conditions, take an uncountable $\Delta$-subsequence where all conditions are isomorphic and find a condition which amalgamates finitely many of these forcing the desired inequality. Thus, the use of forcing allows us to amalgamate a posteriori since the generic filter $G$ takes care of all the possible $\Delta$-systems whose roots belong to $G$ . However in our recursive construction the amalgamations must be done a priori which limits the class of possible amalgamations. In fact since we do a single amalgamation at any given level of $\mathcal{F},$ our spaces tend to be considerably more homogeneous and therefore much easier to analyze.

In Section~\ref{e.bior} we give a proof of Theorem~\ref{theorem.1} and in Section~\ref{k.basis} we prove Theorem~\ref{theorem.2}.

\section{Preliminaries}

We use standard notation for set theory. For $\alpha\in\omega_1$ we denote the $\alpha^{th}$ ordinal and the set $\{\beta: \beta<\alpha\}$.
If $A,B\subset\omega_1$ we say $A<B$ if for all $a\in A$ and $b\in B$, $a<b$.

We follow standard notation for Banach spaces (see, for example, \cite{LTz} and \cite{HMVZ}).
In particular $c_{00}(\omega_1)$ is the vector space of functions
$x:\omega_1\rightarrow\mathbb R$ with finite support (we use $\supp(x)$ for the support of $x$)
If $F$ is a finite subset of $\omega_1$ and $h:F\rightarrow\mathbb R$, we consider the extension of $h$ in $c_{00}(\omega_1)$
to be zero outside of $F$ and still refer to it as $h$ without risk of confusion.
By $e_\alpha$ we denote the function on $\omega_1$ that takes $\alpha$ to 1 and every other $\beta\in\omega_1$ to zero.
For approximation purposes we work most of the time on $c_{00}(\omega_1,\mathbb Q)$, meaning we consider functions in $c_{00}(\omega_1)$ that only take values in $\mathbb Q$.

If $h,x\in c_{00}(\omega_1)$ we denote 
$$\langle h,x\rangle=\sum_{\alpha<\omega_1} h(\alpha)x(\alpha)$$
which is well defined because $x$ and $h$ have finite support.

We recall some notions of Banach space theory relevant for the results.

\begin{defi}
Let $\calX$ be a Banach space and $(y_\alpha,y_\alpha^*)_{\alpha<\omega_1}$ a sequence in $\calX\times\calX^*$.
For $\varepsilon\geq 0$, we say that $(y_\alpha,y_\alpha^*)_{\alpha<\omega_1}$ forms an $\varepsilon$-\emph{biorthogonal system} if
$y_\alpha^*(y_\alpha)=1$ for every $\alpha<\omega_1$, and $|y_\alpha^*(y_\beta)|\leq\varepsilon$ for every $\alpha\neq\beta$. If $\varepsilon=0$ we say $(y_\alpha)_{\alpha<\omega_1}$ forms a \emph{biorthogonal system}.
\end{defi}

A Banach space $\calX$ have the \emph{Mazur intersenction property} (MIP) if every closed convex subset of $\calX$ is the intersection of closed balls. 

The following relates the two previous concepts.

\begin{theor}[\cite{JMS}]
 Let $\calX$ be a Banach space.
\begin{enumerate}
\item If $(y_\alpha,y_\alpha^*)_{\alpha<\kappa}$ forms a biorthogonal system with $\langle y_\alpha^*:\alpha<\kappa\rangle$ dense in $\calX^*$ then $\calX$ admits an equivalent norm with the {\upshape (MIP)}.
\item If $\calX$ is nonseparable and has an equivalent norm with the {\upshape (MIP)}, then $\calX$ has an uncountable $\varepsilon$-biorthogonal system for some $0\leq\varepsilon<1$.
\end{enumerate}
\end{theor}
 
In \cite{Tbaire} it is shown using $\mbox{MA}_{\omega_1}$ plus PID that every uncountable Banach space has an uncountable biorthogonal system. Recall that here $\mbox{MA}_{\omega_1}$ is the standard Baire category principle for the class of compact spaces satisfying the countable chain condition and PID\ is the P-ideal dichotomy stating that for any P-ideal $\mathcal{I}$ of countable subsets of some index set $S$, either $S$ can be partitioned into countably many subsets orthogonal to $\mathcal{I}$ or there is an uncountable subset of $S$ all of whose countable subsets belong to $\mathcal{I}$. For more about this sort of dichotomies the reader is referred to \cite{T2011}.

\begin{defi}
We say that a sequence $(y_\alpha)_{\alpha<\omega_1}$ in a Banach space $\calX$ is an uncountable $K$-basic sequence, for $K\geq 1$,
if for every $\lambda<\omega_1$ and every sequence of reals $(a_\alpha)_{\alpha<\omega_1}$ we have
$$ \left\Vert\sum_{\alpha<\lambda}a_\alpha y_\alpha\right\Vert\leq K \left\Vert\sum_{\alpha<\omega_1} a_\alpha y_\alpha\right\Vert$$
\end{defi}

If a Banach space has an uncountable $K$-basic sequence for some $K\geq 1$ it also contains an uncountable biorthogonal system.

We introduce the main tool of this paper.

\subsection{Capturing Construction Schemes}
Capturing construction schemes where introduced by Stevo Todor\v{c}evi\'{c} in \cite{todor}, where they were used to construct
a compact space with the same properties as the space of \cite{BGT}, a perfect non-metrizable compact convex set and an Asplund space
of the form $C(K)$. In section 8 of \cite{todor} a general framework to construct Banach spaces using construction schemes is introduced.
This framework, together with the forcing amalgamations of \cite{LAT}, constitute the technology behind the proofs of
Theorem~\ref{theorem.1} and Theorem~\ref{theorem.2}. 

\begin{defi}
Let $(m_k,n_k,r_k)_{k<\omega}$ be three sequences of natural numbers. We say that $(m_k,n_k,r_k)_{k<\omega}$ form a \emph{type} if
$m_0=1$, $m_{k-1}>r_k$ and $n_k>k$ for every $k>0$, for every $r<\omega$ there are infinitely many $k$'s with $r_k=r$
and for every $k>0$ we have
$$ m_k=n_k(m_{k-1}-r_k)+r_k$$
\end{defi}

\begin{defi}
Let $\calF\subset[\omega_1]^{<\omega}$, a family of finite subsets of $\omega_1$. 
We say that $\calF$ is a \emph{construction scheme of type } $(m_k,n_k,r_k)_{k}$ if we can partition $\calF=\bigcup_{k<\omega}\calF_k$ and for every $F\in\calF$ there is $R(F)$ initial segment of $F$ with the following properties:

\begin{enumerate}
\item For every $A\subset\omega_1$ finite, there is $F\in\calF$ such that $A\subset F$.
\item $\forall F\in\calF_k$, $|F|=m_k$ and $|R(F)|=r_k$.
\item For all $F,E\in\calF_k$, $E\cap F$ is an initial segment of $F$ and $E$.
\item $\forall F\in\calF_k$, there are unique $F_0,\ldots, F_{n-1}\in\calF_{k-1}$ with
\begin{equation}\label{canon.dec}  F=\bigcup_{i<n} F_i\end{equation}
 
 Furthermore $n=n_k$ and $(F_i)_{i<n_k}$ forms an increasing $\Delta$-system with root $R(F)$, i.e.,
$$	R(F)< F_0\setminus R(F) <\ldots < F_{n_k-1}\setminus R(F)$$
we call this the \emph{canonical decomposition} of $F$.
\end{enumerate}
\end{defi}

For $F\in\calF$ we call the \emph{rank} of $F$ the natural $k$ such that $F\in\calF_k$.
We assume also that $\calF_0$ are all singletones in $\omega_1$. We use the elements $F$ of $\calF$ to ``approximate'' an uncountable structure in $\omega_1$ and use \eqref{canon.dec} in the recursive construction.
For this we want all approximations of the same rank $k$ to be ``isomorphic''.

Given a type $(m_k,n_k,r_k)_{k<\omega}$ there is a construction scheme $\calF$ of this type (see proof of Theorem 4.8 of \cite{todor}).

Throughout the rest of the paper we use $k$ for the rank of $F\in\calF$ and $m_k, n_k$ and $r_k$ as above and omit reference to the type of a construction scheme.

If $(D_\alpha:\alpha<\omega_1)$ is a sequence of finite subsets of $\omega_1$, we say it is a $\Delta$-system with root $R$ if
for every $\alpha<\beta<\omega_1$ we have $R=D_\alpha\cap D_\beta$ and $R<D_\alpha\setminus R< D_\beta\setminus R$.

Let $F\in\calF$ and $F=\bigcup_{i<n_k}F_i$ be the canonical decomposition of $F$.
Then, all $F_i$'s have the same size, and we denote $\varphi_i:F_0\rightarrow F_i$ the unique increasing bijection between $F_0$ and $F_i$.
If $f$ is a map on $F_0$, we define the map $\varphi_i(f)$ on $F_i$ as $f\circ\varphi_i^{-1}$. 

\begin{defi}
A construction scheme $\calF$ is \emph{capturing} if for every $n<\omega$ and every uncountable 
$\Delta$-System $(D_\alpha:\alpha<\omega_1)$ with root $R$,
we can find $F\in\calF$ and $\alpha_0<\ldots<\alpha_{n-1}$ such that 
\begin{gather*}
R\subset R(F)\\
D_{\alpha_i} \subset F_i \quad\mbox{for all }i<n.\\
\varphi_i(D_0)=D_i\quad\mbox{for all }i<n.
\end{gather*}
where $(F_i)_{i<n_k}$ is the canonical decomposition of $F$. In particular $n>n_k$.
\end{defi}

The existence of capturing construction schemes follows from $\diamondsuit$ (Theorem 4.8 of \cite{todor})
By Theorem~\ref{theorem.1} and a result of \cite{Tbaire} we obtain the following

\begin{coro}[$\mbox{MA}_{\omega_1}+\mbox{PID}$]
There are no capturing construction schemes.
\end{coro} 

It is shown in \cite{LT} that there is a Souslin tree if there is a 3-capturing construction scheme.

\begin{coro}[MA$_{\omega_1}$]
There are no 3-capturing construction schemes.
\end{coro} 

We give an overview of the proof of Theorem~\ref{theorem.1} and Theorem~\ref{theorem.2} (see also Section 8 of \cite{todor}).

\subsection{Overview of proof}

The construction of the Banach spaces $\Xe$ and $\Xk$ will follow an abstract approach for producing nonseparable Banach structures. 

We start with a capturing construction scheme $\calF$.
First, we construct (recursively) a family $\calH=\bigcup_{F\in\calF}\calH_F$ where $\calH_F$ are functions $f:F\rightarrow [0,1]\cap\mathbb Q$.
For $\Xe$ we will have $\calH_F=\{h_\alpha^F:\alpha\in F\}$. To guarantee nonseparability we want to have the following condition
\begin{equation}\label{nonseparability}
 h_\alpha^F\restriction\alpha=0\qquad h_\alpha^F(\alpha)=1
\end{equation}

The role of $\calH$ is to be a norming set, for that we need the following coherence conditions
\begin{gather}\label{coherence.1}
\forall F,E\in\calF\quad\mbox{if $E\subset F$ then}\quad h_\alpha^F\restriction E=h_\alpha^F\quad\forall\alpha\in E \\ 
\label{coherence.2} \forall F,E\in\calF\quad\mbox{if $E\subset F$ then}\quad f\restriction E\in\conv(\pm\calH_E)\quad\forall f\in\calH_F
\end{gather}
Let $\calH_k=\bigcup_{i<k,F\in\calF_i}\calH_F$.
Suppose $\calH_k$ has been defined and $F\in\calF$ has rank $k$. Let $F=\bigcup_{i<n_k}F_i$ the the canonical decomposition of F.
We will define $\calH_F$ by amalgamating the elements of $\calH_{F_i} (i<n_k)$
in such a way that \eqref{nonseparability}, \eqref{coherence.1} and \eqref{coherence.2} holds for $\Xe$ and \eqref{coherence.2} for $\Xk$.

This concludes the construction of $\calH$. Next, we will define $\Vert\cdot\Vert$ in $c_{00}(\omega_1)$
\begin{equation}\label{norm}
\Vert x\Vert=\max\{|\langle f, x\rangle|: f\in\calH\}
\end{equation} 

Note that $\Vert\cdot\Vert$ is well defined by \eqref{coherence.1} and \eqref{coherence.2} and by \eqref{nonseparability}
we have $\Vert x\Vert =0$ if and only if $x=0$ (this for the construction of $\Xe$, for $\Xk$ the vectors $e_\alpha\in\calH$ for every $\alpha<\omega_1$) so it defines a norm on $c_{00}(\omega_1)$.
The respective Banach space $\calX$ will be the completion of $(c_{00}(\omega_1),\Vert\cdot\Vert)$.

To prove that $\calX$ has indeed the properties that we want we will use the capturing of $\calF$.
Arguing by contradiction we take an uncountable sequence $(y_\alpha)_{\alpha<\omega_1}$ in $\calX$ with a certain property.
We show (following \cite{LAT}) that there is an inequality that uncountably many $y_\alpha$'s satisfy.

Take $(x_\alpha)_{\alpha<\omega_1}$ in $c_{00}(\omega_1,\mathbb Q)$ approximating the $y_\alpha$'s
and apply the $\Delta$-System lemma and a counting argument (this is why we take $\mathbb Q$ instead of $\mathbb R$)
to obtain $\Gamma\subset\omega_1$ uncountable such that
\begin{enumerate}
\item $(\supp(x_\alpha): \alpha\in\Gamma)$ forms a $\Delta$-System and
\item the $x_\alpha$'s are ``isomorphic'' in some manner.
\end{enumerate}

Finding $F\in\calF$ capturing enough $x_\alpha$'s we can construct vectors that contradict the inequality.

\section{Proof of Theorem \ref{theorem.1}}\label{e.bior}

Let $\calF$ be a capturing construction scheme and $0<\varepsilon<1$ rational.
$\calH_1$ is form by $h_\alpha^{\{\alpha\}}$ taking values in $\{\alpha\}$ and sending $\alpha\mapsto 1$.

Suppose $\calH_k$ has been built satisfying \eqref{nonseparability}, \eqref{coherence.1} and \eqref{coherence.2}.
Let $F\in\calF_k$ and $F=\bigcup_{i<n_k}F_i$ the canonical decomposition of $F$.
Then, we let $\calH_F=\{h_\alpha^F:\alpha\in F\}$ where $h_\alpha^F$ is define in the following way
\begin{enumerate}
\item For $\alpha\in R$, define $h_\alpha^F := h_0^{F_0}+\sum_{0<i<n_k}\varphi_i(h_\alpha^{F_0})\restriction (F_i\setminus F_0) $.
\item For $\alpha\in F_0\setminus R$, define 
                $$h_\alpha^F := h_\alpha^{F_0}+
                                         \varepsilon\sum_{2\leq i<n_k}(-1)^i\varphi_i(h_\alpha^{F_0})\restriction (F_i\setminus F_0). $$
\item For $\delta\in F_1\setminus R$, and $\alpha\in F_0\setminus R$ with $\varphi_1(\alpha)=\delta$, define 
                $$h_\delta^F :=\varphi_1(h_\alpha^{F_0})+
                                                         \varepsilon\sum_{2\leq i<n_k}(-1)^{i+1}\varphi_{i}(h_\alpha^{F_0})\restriction (F_i\setminus F_0).$$
\item For $\alpha\in F_j\setminus R$ with $2\leq j<n_k$, define $h_\alpha^F=h_\alpha^{F_j}$.
\end{enumerate}

It is clear that $\calH_{k+1}$ satisfies \eqref{nonseparability} and \eqref{coherence.1}.
Note that if $E\in\calF$ is contained in $F$ and $\alpha\in F$, there is $f\in\calH_E$ such that 
 $h_\alpha^F(\gamma)$ equals either $f(\gamma)$ or $\varepsilon f(\gamma)$ for every $\gamma\in E$.
This shows that \eqref{coherence.2} holds for $\calH_{k+1}$.
The same observation shows \begin{equation}\label{ebiorthogonal}
|h_\alpha^F(e_\beta)|\leq\varepsilon
\end{equation} for all $\alpha\neq\beta$ in $F$. 

This finishes the construction of $\calH$.

Define the norm $\Vert\cdot\Vert_\varepsilon$ as in \eqref{norm} and
let $\Xe$ be the completion of $(c_{00}(\omega_1),\Vert\cdot\Vert_\varepsilon)$. 

We check that $\Xe$ is as we wanted.
Define $h_\alpha$ to be the union of all $(h_\alpha^F:F\in\calF)$ which is well defined by \eqref{coherence.1}.
By \eqref{ebiorthogonal} the sequence $(e_\alpha,h_\alpha)_{\alpha<\omega_1}$ forms an un uncountable $\varepsilon$-biorthogonal system.

Suppose $(y_\alpha,y^*_\alpha)_{\alpha<\omega_1}$ is a $\tau$-biorthogonal system for
$0\leq\tau<\frac{\varepsilon}{1+\varepsilon}$. We can assume that the $y_\alpha$'s are normalized.

\begin{lemma}\label{lem1}
                There is $\Gamma\subset\omega_1$ uncountable and $\delta>0$ such that, for every $n,m<\omega$ with
                $\frac{m}{2n}=\varepsilon$ and every $\alpha_0<\ldots<\alpha_{2n+1}$ we have,
                \begin{equation}\label{condition}
                \left\Vert (y_{\alpha_0}-y_{\alpha_1})-\frac{1}{m}\sum_{i=1}^n(y_{\alpha_{2i}}-y_{\alpha_{2i+1}})\right\Vert_\varepsilon\geq\delta
                \end{equation}
\end{lemma}

\begin{proof}
Let $N<\omega$ and $\Gamma\subset\omega_1$ uncountable such that
$$ \sup_{\alpha\in\Gamma}\Vert y^*_\alpha\Vert\leq N $$
Then 
{\footnotesize\begin{align*}
\left\Vert (y_{\alpha_0}-y_{\alpha_1})-\frac{1}{m}\sum_{i=1}^n(y_{\alpha_{2i}}-y_{\alpha_{2i+1}})\right\Vert_\varepsilon
                &\geq\left|\frac{f_{\alpha_1}}{N}\left((y_{\alpha_0}-y_{\alpha_1})
                -\frac{1}{m}\sum_{i=1}^n(y_{\alpha_{2i}}-y_{\alpha_{2i+1}})\right)\right| \\
                &\geq\frac{1}{N}\Bigl(1-\tau-\frac{1}{m}(2n\tau)\Bigr)=\frac{1}{N}\Bigl(1-\tau(1+\frac{2n}{m})\Bigr)
\end{align*}}
Taking $\delta=\frac{1}{N}(1-\tau(1+\frac{2n}{m}))=\frac{1}{N}(1-\tau\frac{1+\varepsilon}{\varepsilon})>0$ we obtain the result.
\end{proof}

Theorem~\ref{theorem.1} follows if we show that

\begin{lemma}\label{lem2}
              For every normalized $(y_\alpha)_{\alpha\in\Gamma}$ in $\Xe$, there is $m,n<\omega$ with
                                                        $\frac{m}{2n}=\varepsilon$ and $\alpha_0<\ldots<\alpha_{2n+1}$ such that
                                                        $$
                                                        \left\Vert (y_{\alpha_0}-y_{\alpha_1})-\frac{1}{m}\sum_{i=1}^n(y_{\alpha_{2i}}-y_{\alpha_{2i+1}})\right\Vert_\varepsilon<\delta
                                                        $$
\end{lemma}
\begin{proof}
Let $m$ and $n$, big enough so that $1/m<\delta/2$ and $m/2n=\varepsilon$.

Let $x_\alpha\in c_{00}(\omega_1,\mathbb Q)$ for $\alpha\in\Gamma$ normalized such that
$$\Vert y_\alpha-x_\alpha\Vert_\varepsilon<\frac{\delta}{4(n+1)}\quad\mbox{for every }\alpha\in\Gamma.$$

Note that
\begin{align*}
\left\Vert (y_{\alpha_0}-y_{\alpha_1})-\frac{1}{m}\sum_{i=1}^n  (y_{\alpha_{2i}}-y_{\alpha_{2i+1}})\right\Vert_\varepsilon
\leq\qquad\qquad\qquad& \\
\leq\left\Vert (x_{\alpha_0}-x_{\alpha_1})-\frac{1}{m}\sum_{i=1}^n(x_{\alpha_{2i}}-x_{\alpha_{2i+1}})\right\Vert_\varepsilon
&+\sum_{i=0}^{2n+1}\Vert y_\alpha-x_\alpha\Vert_\varepsilon \\
\leq \left\Vert (x_{\alpha_0}-x_{\alpha_1})-\frac{1}{m}\sum_{i=1}^n(x_{\alpha_{2i}}-x_{\alpha_{2i+1}})\right\Vert_\varepsilon
&+\frac{\delta}{2}
\end{align*}
thus, it is enough to find $\alpha_0<\alpha_1<\ldots<\alpha_{2n+1}$ in $\Gamma$ such that 

\begin{equation}\label{eq.1}
\left\Vert (x_{\alpha_0}-x_{\alpha_1})-\frac{1}{m}\sum_{i=1}^n(x_{\alpha_{2i}}-x_{\alpha_{2i+1}})\right\Vert_\varepsilon<\frac{\delta}{2}
\end{equation}

Apply the $\Delta$-System lemma and a counting argument to find $\Gamma_0\subset\Gamma$ uncountable such that

\begin{enumerate}
\item Let $D_\alpha=\supp(x_\alpha)$, then the collection $(D_\alpha:\alpha\in\Gamma_0)$ form a $\Delta$-System with
       $|D_\alpha|=|D_\beta|=d$ for every $\alpha,\beta\in\Gamma_0$.
\item For $\alpha,\beta\in\Gamma_0$ and $\varphi_{\alpha,\beta}:D_\alpha\rightarrow D_\beta$ an increasing bijection then
       $x_\beta=\varphi_{\alpha,\beta}(x_\alpha)$.
\end{enumerate}

Since $\calF$ is capturing there is $F\in\calF$ and some $\alpha_0<\ldots<\alpha_{2n+1}$ in $\Gamma_0$, such that $F$ captures 
$(D_{\alpha_i}: i\leq 2n+1)$. Let 

$$ w=  (x_{\alpha_0}-x_{\alpha_1})-\frac{1}{m}\sum_{i=1}^n(x_{\alpha_{2i}}-x_{\alpha_{2i+1}})$$

Note that $w\restriction R(F)$ is identically zero. We show that $\Vert w\Vert <\delta/2$. Let $f\in\calH_F$. 

If $f$ is of the form (1) it is clear that $\langle f,w\rangle =0$.

If $f$ is of the form (2) then $f=h_\alpha^F$ for some $\alpha\in F$ and
\begin{align*}
\langle f,w\rangle=h_\alpha^{F_0}(x_{\alpha_0})-\frac{\varepsilon}{m}\sum_{i=2}^n(h_\alpha^{F_0}(x_{\alpha_0})+h_\alpha^{F_0}(x_{\alpha_0}))
=h_\alpha^{F_0}(x_{\alpha_0})\Bigl(1-\varepsilon\frac{2n}{m}\Bigr) =0
\end{align*}
because the amalgamation for $f$ nullifies the term in $\alpha_1$ and changes the sign of the other odd terms.

If $f$ is of the form (3) then
\begin{align*}
\langle f,w\rangle=-h_\alpha^{F_1}(x_{\alpha_1})+\frac{\varepsilon}{m}\sum_{i=2}^n(h_\alpha^{F_1}(x_{\alpha_1})+h_\alpha^{F_1}(x_{\alpha_1}))= h_\alpha^{F_1}(x_{\alpha_1})\Bigl(\varepsilon\frac{2n}{m}-1\Bigr) =0
\end{align*}
because the amalgamation for $f$ nullifies the term in $\alpha_0$ and changes the sign the other even terms

Finally if $f$ is of the form (4) then 
$|\langle f,w\rangle|=|\frac{1}{m}\langle h_\alpha^{F_j},\varphi_{\alpha_j}(z)\rangle|\leq\frac{1}{m}<\delta/2$
as we wanted to show. Thus, $w$ witnesses \eqref{eq.1} contradicting Lemma~\ref{lem1} and finishing the proof.
\end{proof}

 
\section{Proof of Theorem \ref{theorem.2}}\label{k.basis}

We construct $\calH$ by recursion. For $\Xk$ the collection $\calH_F$ will have the following \emph{closure property}:
\begin{equation}\label{closureprop}
\forall f\in\calH_F,\ \delta\in F\quad \lambda^{-1}(f\restriction\delta)\in\calH_F
\end{equation}

Let $\calF$ be a capturing construction scheme and let $K>1$.
$\calH_1$ is form of functions of the form $K^{-n}e_\alpha$ for every $\alpha<\omega_1$ and $n<\omega$.

Suppose $\calH_k$ has been constructed satisfying \eqref{closureprop} and \eqref{coherence.2}.
Let $F\in\calF_k$ and $F=\bigcup_{i<n_k}F_i$ 
be the canonical decomposition of $F$. Then, we let $\calH_F$ be the collection of functions of the following type:
\begin{enumerate}
\item $e_\alpha$, for $\alpha\in F$.
\item $\sum_{i<n_k}\varphi_i(f)\restriction (F_i\setminus F_0)$ for every $f\in \calH_{F_0}$.
\item $\frac{1}{K^n}\left(\sum_{i<n_k}\varphi_i(f)\restriction (F_i\setminus F_0)\right)\restriction\delta$ for every $f\in\calH_{F_0}$, every $\delta\in F$ and $n=1,2\ldots$
\end{enumerate}

It is clear that \eqref{closureprop} and \eqref{coherence.2} holds for $\calH_{k+1}$. This finishes the construction of $\calH$.

Define $\Vert\cdot\Vert_K$ as in \eqref{norm} and let $\Xk$ be the completion of $(c_{00}(\omega_1),\Vert\cdot\Vert_K)$.

We see that $\Xk$ is as we wanted. We first show that $\Xk$ has an uncountable $K$-basic sequence. Let $(e_\alpha)_{\alpha<\omega_1}$ be the canonical unit vector basis.

\begin{lemma}
The vectors $(e_\alpha)_{\alpha<\omega_1}$ form a normalized $K$-basis of $\Xk$. In particular $\Xk$ is not separable.
\end{lemma}
\begin{proof}
It is clear that the $e_\alpha$'s are normalized. To see they are a $K$-basic sequence let $n<m<\omega$, $\alpha_1<\ldots<\alpha_m<\omega_1$
and $(a_i)_{i=1}^m\in\mathbb R^m$. Let $F\in\calF$ such that $\alpha_i\in F$ for $i=1,\ldots,m$. Take $\delta=\alpha_{n+1}$ and $f\in\calH_F$ such that $\sum_{i=1}^na_ie_{\alpha_i}$ attains the norm at $f$.

If $f$ is of the form (1) then $f\restriction\delta=Kg$ for some $g\in\calH_F$ and if $f$ is of the form (2) then $f=g$ for some $g\in\calH_F$. Thus,
\begin{align*}
\left\Vert\sum_{i=1}^n a_ie_{\alpha_i}\right\Vert_K&=\left|\left\langle f,\sum_{i=1}^na_ie_{\alpha_i}\right\rangle\right|
=\left|\left\langle f\restriction\delta,\sum_{i=1}^ma_ie_{\alpha_i}\right\rangle\right| \\
&\leq K\left|\left\langle g,\sum_{i=1}^na_ie_{\alpha_i}\right\rangle\right|\leq K\left\Vert\sum_{i=1}^m a_ie_{\alpha_i}\right\Vert_K
\end{align*} 
as we wanted to show.
\end{proof}

We proceed by contradiction. Suppose now that $(y_\alpha)_{\alpha<\omega_1}$ is a $K'$-basic sequence with $1\leq K'<K$.
Fix $K'<L<K$ and let $n<\omega$ such that
\begin{equation}\label{condi.1}
\frac{1}{K}+\frac{1}{n}<\frac{1}{L}
\end{equation}
Take a normalized sequence $(x_\alpha)_{\alpha<\omega_1}$ in $c_{00}(\omega_1,\mathbb Q)$ such that
$$
\Vert x_\alpha-y_\alpha\Vert_K <\min\Bigl\{\frac{1}{4K'n},\frac{L-K'}{8(K')^2n}\Bigr\} \quad\mbox{ for every }\alpha<\omega_1.
$$
The following lemma plays the same role of Lemma~\ref{lem1} in Theorem~\ref{theorem.1}
\begin{lemma}\label{lemma1} For every $\alpha_1<\ldots<\alpha_{2n}<\omega_1$
\begin{equation*}
\left\Vert \sum_{i=1}^n x_{\alpha_i}\right\Vert_K \leq L\left\Vert\sum_{i=1}^n x_{\alpha_i} -\sum_{i=n+1}^{2n} x_{\alpha_i}\right\Vert_K
\end{equation*}
\end{lemma}

\begin{proof}
Note first that $\Vert\sum_{i=1}^n x_{\alpha_i} -\sum_{i=n+1}^{2n} x_{\alpha_i}\Vert_K\geq 1/2K'$. Indeed, suppose otherwise then
\begin{align*}
1=\Vert y_{\alpha_1}\Vert_K&\leq K'\left\Vert\sum_{i=1}^ny_{\alpha_i}-\sum_{i=n+1}^{2n}y_{\alpha_i}\right\Vert_K \\
&\leq K'\left\Vert\sum_{i=1}^n x_{\alpha_i} -\sum_{i=n+1}^{2n} x_{\alpha_i}\right\Vert_K+K'\sum_{i=1}^{2n}\Vert y_{\alpha_i}-x_{\alpha_i}\Vert_K\\ &< K'\left(\frac{1}{2K'}+\frac{2n}{4K'n}\right)=1
\end{align*}

Now
\begin{align*}
\left\Vert\sum_{i=1}^n x_{\alpha_i}\right\Vert_K&\leq\left\Vert\sum_{i=1}^n y_{\alpha_i}\right\Vert_K+\sum_{i=1}^n \Vert x_{\alpha_i}-y_{\alpha_i}\Vert_K \\
&\leq K'\left\Vert\sum_{i=1}^ny_{\alpha_i}-\sum_{i=n+1}^{2n}y_{\alpha_i}\right\Vert_K +\sum_{i=1}^{n}\Vert x_{\alpha_i}-y_{\alpha_i}\Vert_K\\
&\leq K'\left\Vert\sum_{i=1}^n x_{\alpha_i} -\sum_{i=n+1}^{2n} x_{\alpha_i}\right\Vert_K+2K'\sum_{i=1}^{2n}\Vert x_{\alpha_i}-y_{\alpha_i}\Vert_K\\
&\leq K'\left\Vert\sum_{i=1}^n x_{\alpha_i} -\sum_{i=n+1}^{2n} x_{\alpha_i}\right\Vert_K+4K'n\frac{L-K'}{8(K')^2n}\\
&\leq K'\left\Vert\sum_{i=1}^n x_{\alpha_i} -\sum_{i=n+1}^{2n} x_{\alpha_i}\right\Vert_K+
(L-K')\left\Vert\sum_{i=1}^n x_{\alpha_i} -\sum_{i=n+1}^{2n} x_{\alpha_i}\right\Vert_K
\end{align*}
which is what we wanted to prove.
\end{proof}
We want to use the capturing of $\calF$ to contradict the lemma above.

We proceed as before and find $\Gamma\subset\omega_1$ uncountable such that

\begin{enumerate}
\item If $D_\alpha=\supp(x_\alpha)$, then the collection $(D_\alpha:\alpha\in\Gamma)$ form a $\Delta$-System with
       $|D_\alpha|=|D_\beta|=d$ for every $\alpha,\beta\in\Gamma$.
\item There is a function $z:d\rightarrow\mathbb Q$ such that, if $\varphi_\alpha:d\rightarrow D_\alpha$ is the unique order increasing
       bijection, then $x_\alpha=\varphi_\alpha(z)$
\end{enumerate}

Since $\calF$ is capturing, there is $F\in\calF$ and $\alpha_1<\ldots<\alpha_{2n}<\omega_1$ in $\Gamma$ such that $F$
captures $(D_{\alpha_i}:i=1,\ldots,2n)$.

Let
$$v=\sum_{i=1}^n x_{\alpha_i}\quad\mbox{and}\quad w=\sum_{i=1}^n x_{\alpha_i}-\sum_{i=n+1}^{2n}x_{\alpha_i}$$

We show that $\Vert v\Vert_K>L\Vert w\Vert_K$. Let $F=\bigcup_{i<n_k}F_i$ be the canonical decomposition of $F$.
Since the $x_{\alpha_i}$'s are normalized there is $h\in\calH_{F_0}$ such that $|\langle h,x_{\alpha_1}\rangle|=1$.
Taking $f=\sum_{i<n_k}\varphi_i(h)$ we get $|\langle f,v\rangle|=n$. Thus $\Vert v\Vert_K\geq n$.

Take now $f\in\calH_F$.

If $f$ is of the form (1) then, $|\langle f, w\rangle|=0$

If $f$ is of the form (2) then, $f=(1/K)\sum_{i<n_k}\varphi_i(h)\restriction\delta$ for some $\delta\in F$ and $h\in\calH_{F_0}$.
If $\delta\in R(F)$ then $|\langle f, w\rangle|=0$. Suppose $\delta\in F_j\setminus R(F)$ and $\eta\in F_{0}$ is such that
$\varphi_j(\eta)=\delta$

Suppose $j<n$ then
\begin{align*}
|\langle f, w\rangle| &\leq \left|\frac{1}{K}\langle \sum_{i<j}\varphi_i(h),w\rangle\right|+
\frac{1}{K}|\langle h\restriction\eta, x_{\alpha_1}\rangle| \\
&\leq \frac{n-1}{K}+\Vert x_{\alpha_0}\Vert_K=\frac{n-1}{K}+1<\frac{n}{L}\leq\frac{1}{L}\Vert v\Vert_K
\end{align*}
by \eqref{condi.1}.

Suppose now $j\geq n$. Then
{\footnotesize\begin{align*}
|\langle f,w\rangle|&\leq\frac{1}{K}\left|\sum_{i<n-1}\langle\varphi_i(h),x_{\alpha_i}\rangle+ \langle\varphi_{n-1}(h),x_{\alpha_{n-1}}\rangle-\sum_{n\geq i<j}\langle\varphi_i(h),x_{\alpha_i}\rangle - \langle\varphi_{j}(h)\restriction\delta,x_{\alpha_{j}}\rangle\right|\\
&\leq \frac{1}{K}|(n-1)+\langle h,x_{\alpha_0}\rangle-(j-n)-\langle h\restriction\eta,x_{\alpha_0}\rangle\\
&\leq \frac{n-1}{K}+\frac{\Vert x_{\alpha_0}\Vert_K+\Vert x_{\alpha_0}\restriction\eta\Vert_K}{K}
\leq\frac{n}{K}+1<\frac{n}{L}\leq \frac{1}{L}\Vert v\Vert_K
\end{align*}}

If $f$ is of the form (3) then $|\langle f, w\rangle|\leq 1\leq\frac{n}{K}+1<\frac{n}{L}\leq\frac{1}{L}\Vert v\Vert_K$.

We conclude that $\Vert w\Vert_K<\frac{1}{L}\Vert v\Vert_K$ but this contradicts Lemma~\ref{lemma1} and thus $\Xk$ is as we wanted.


\begin{thebibliography}{10}
\parskip=3mm

\bibitem{BGT}{
M. Bell, J. Ginsburg, and S. Todor\v{c}evi\'{c}, \emph{Countable spread of $exp\ Y$ and $\lambda Y$}, Top. Appl. {14} (1982) 1--12
}

\bibitem{HMVZ}{
Hajek, Petr; Montesinos Santaluc\'{i}a, Vicente; Vanderwerff, Jon; Zizler, V\'{a}clav, Biorthogonal systems in Banach spaces. CMS Books in Mathematics/Ouvrages de Mathématiques de la SMC, 26. Springer, New York, 2008. xviii+339 pp. ISBN: 978-0-387-68914-2 
}

\bibitem{JMS}{
M. Jimenez Sevilla and J.P. Moreno, \emph{Renorming Banach Spaces with the Mazur Intersection Property}, J. Funct. Anal. {144} (1997), 486--504.
}

\bibitem{LTz}{
 Lindenstrauss, Joram; Tzafriri, Lior, Classical Banach spaces. I. Sequence spaces. Ergebnisse der Mathematik und ihrer Grenzgebiete, Vol. 92. Springer-Verlag, Berlin-New York, 1977. xiii+188 pp. ISBN: 3-540-08072-4
} 

\bibitem{LAT}{
J. L\'{o}pez-Abad and S. Todor\v{c}evi\'{c}, \emph{Generic Banach spaces and generic simplexes}, J. Funct. Anal. {261} (2011), 300--386.
}


\bibitem{LT}{
Fulgencio L\'{o}pez and Stevo Todor\v{c}evi\'{c}, \emph{Trees and gaps from a construction scheme}, preprint 2015.
}

\bibitem{Tbaire}{
Stevo Todor\v{c}evi\'{c}, \emph{Biorthogonal systems and quotient spaces via Baire Category methods}, Math. Annalen 335 (2006), 687--715.
}

\bibitem{T2011}{
Todorcevic, Stevo, \emph{Combinatorial dichotomies in set theory}. Bull. Symbolic Logic 17 (2011), no. 1, 1--72. 
}

\bibitem{todor}{
Stevo Todor\v{c}evi\'{c}, \emph{A construction scheme for non-separable structures}, preprint 2014.
}

\end{thebibliography}
\end{document}